\documentclass{amsart}
\usepackage{amsmath}
\usepackage{amssymb}
\usepackage{graphicx}
\usepackage{amscd}

\newtheorem{theorem}{Theorem}
\theoremstyle{plain}
\newtheorem{acknowledgement}{Acknowledgement}

\newtheorem{corollary}{Corollary}

\newtheorem{lemma}{Lemma}

\newtheorem{proposition}{Proposition}
\newtheorem{remark}{Remark}

\numberwithin{equation}{section}

\input{tcilatex}

\begin{document}
\title[Nonexistence of Levi flat hypersurfaces]{Nonexistence of Levi flat
hypersurfaces with positive normal bundle in compact K\"{a}hler manifolds of
dimension $\geqslant 3$}
\author{S\'{e}verine Biard}
\address{Institut de Math\'{e}matiques \\
UMR 7586 du CNRS, case 247 \\
Universit\'{e} Pierre et Marie-Curie \\
4 Place Jussieu \\
75252 Paris Cedex 05\\
France\\
Current address: LAMAV, \\
Universit\'{e} Polytechnique Hauts-de-France, \\
Campus du Mont Houy, \\
59313 Valenciennes Cedex 9\\
France}
\email{severine.biard@uphf.fr}
\author{Andrei Iordan}
\address{Sorbonne Universit\'{e}\\
Facult\'{e} des Sciences et Ing\'{e}nierie\\
Institut de Math\'{e}matiques de Jussieu-Paris Rive Gauche \\
4 Place Jussieu \\
75252 Paris Cedex 05\\
France}
\email{andrei.iordan@imj-prg.fr}
\date{September, 30, 2019}
\subjclass{32V40, 32F32, 32Q15, 32W05}
\keywords{Levi flat hypersurface, weighted $\overline{\partial }$-equation}
\dedicatory{In memory of Gennadi M.\ Henkin}

\begin{abstract}
Let $X$ be a compact connected K\"{a}hler manifold of dimension $\geqslant 3$
and $L$ a $C^{\infty }$ Levi flat hypersurface in $X$. Then the normal
bundle to the Levi foliation does not admit a Hermitian metric with positive
curvature along the leaves. This represents an answer to a conjecture of
Marco Brunella.
\end{abstract}

\maketitle

\section{\protect\bigskip Introduction}

A classical theorem of Poincar\'{e}-Bendixson \cite{Poincare1881}, \cite%
{Poincare1882}, \cite{Bendixon1901} states that every leaf of a foliation of
the real projective plane accumulates on a compact leaf or on a singularity
of the foliation. As a holomorphic foliation $\mathcal{F}$ of codimension $1$
of $\mathbb{CP}_{n}$, $n\geqslant 2$, does not contain any compact leaf and
its singular set $Sing~\mathcal{F}$ is not empty, a major problem in
foliation theory is the following: can $\mathcal{F}$ contain a leaf $F$ such
that $\overline{F}\cap Sing~\mathcal{F}=\emptyset $? If this is the case,
then there exists a nonempty compact set $K$ called exceptional minimal,
invariant by $\mathcal{F}$ and minimal for the inclusion such that $K\cap
Sing~\mathcal{F}=\emptyset $. The problem of the existence of an exceptional
minimal in $\mathbb{CP}_{n}$, $n\geqslant 2$ is implicit in \cite{Camacho88}.

In \cite{Cerveau93} D. Cerveau proved a dichotomy under the hypothesis of
the existence of a holomorphic foliation $\mathcal{F}$ of codimension $1\ $%
of $\mathbb{CP}_{n}$ which admits an exceptional minimal $\mathfrak{M}$: $%
\mathfrak{M}$ is a real analytic Levi flat hypersurface in $\mathbb{CP}_{n}$
(i. e. $T\left( \mathfrak{M}\right) \cap JT\left( \mathfrak{M}\right) $ is
integrable, where $J$ is the complex structure of $\mathbb{CP}_{n}$), or
there exists $p\in \mathfrak{M}$ such that the leaf through $p$ has a
hyperbolic holonomy and the range of the holonomy morphism is a linearisable
abelian group. This gave rise to the conjecture of the nonexistence of
smooth Levi flat hypersurface in $\mathbb{CP}_{n}$, $n\geqslant 2$.

The conjecture was proved for $n\geqslant 3$ by A.\ Lins Neto \cite{LinsNeto}
\ for real analytic Levi flat hypersurfaces and by Y.-T.\ Siu \cite{Siu00}
for $C^{12}$ smooth Levi flat hypersurfaces. The methods of proofs for the
real analytic case are very different from the smooth case.

A real hypersurface of class $C^{2}$ in a complex manifold is Levi flat if
its Levi form vanishes or equivalently, it admits a foliation by complex
hypersurfaces. We say that a (non-necessarly smooth) real hypersurface $L$
in a complex manifold $X$ is Levi flat if $X\backslash L$ is pseudoconvex.
An example of (non-smooth) Levi flat hypersurface in $\mathbb{CP}_{2}$ is $%
L=\left\{ \left[ z_{0},z_{1},z_{2}\right] :\ \left\vert z_{1}\right\vert
=\left\vert z_{2}\right\vert \right\} $, where $\left[ z_{0},z_{1},z_{2}%
\right] $ are homogeneous coordinates in $\mathbb{CP}_{2}$ (see \cite%
{Henkin00}).

In \cite{Iordan08} Iordan and Matthey proved the nonexistence of Lipschitz
Levi flat hypersurfaces in $\mathbb{CP}_{n}$, $n\geqslant 3$, which are of
Sobolev class $W^{s}$, $s>9/2$. A principal element of the proof is that the
Fubini-Study metric induces a metric of positive curvature on any quotient
of the tangent space.

Nonexistence questions for the Levi flat hypersurfaces in compact K\"{a}hler
manifolds were first discussed by T. Ohsawa in \cite{Ohsawa07}, who proved
the nonexistence of real-analytic Levi flat hypersufaces with Stein
complement in compact K\"{a}hler manifolds of dimension $\geqslant 3$.

In \cite{Brunella08}, M. Brunella proved that the normal bundle to the Levi
foliation of a closed real analytic Levi flat hypersurface in a compact K%
\"{a}hler manifold of dimension $n\geqslant 3$ does not admit any Hermitian
metric with leafwise positive curvature. The real analytic hypothesis may be
relaxed to the assumption of $C^{2,\alpha }$, $0<\alpha <1$, such that the
Levi foliation extends to a holomorphic foliation in a neighborhood of the
hypersurface.

The main step in his proof is to show that the existence of a Hermitian
metric with leafwise positive curvature on the normal bundle to the Levi
foliation of a compact Levi flat hypersurface $L$ in a Hermitian manifold $X$%
, implies that $X\backslash L$ is strongly pseudoconvex, i.e. there exists
on $X\backslash L$ an exhaustion function which is strongly plurisubharmonic
outside a compact set. This was generalized in \cite{Brunella2010} for
invariant compact subsets of a holomorphic foliation of codimension one. Of
course, if $X$ is the complex projective space, then every proper
pseudoconvex domain in $X$ is Stein \cite{Takeushi67}.

Brunella stated also the following conjecture \cite{Brunella08}: Let $X$ be
a compact connected K\"{a}hler manifold of dimension $n\geqslant 3$ and $L$
a $C^{\infty }$ compact Levi flat hypersurface in $X$. Then the normal
bundle to the Levi foliation does not admit any Hermitian metric with
leafwise positive curvature.

The assumption $n\geqslant 3$ is necessary in this conjecture (see Example
4.2 of \cite{Brunella08}).

In \cite{Brunella2011} Brunella and Perrone proved that every leaf of a
holomorphic foliation $\mathcal{F}$ of codimension one of a projective
manifold $X$ of dimension at least $3$ and such that $Pic\left( X\right) =%
\mathbb{Z}$ accumulates on the singular set of the foliation. In this case
the normal bundle to the foliation is ample.

In \cite{Ohsawa2013}, T. Ohsawa considered a $C^{\infty }$ Levi flat compact
hypersurface $L$ in a compact K\"{a}hler manifold $X$ such that the normal
bundle to the Levi foliation admits a fiber metric whose curvature is
semipositive of rank$\geqslant k$ on the holomorphic tangent space to the
leaves and proved that $X\backslash L$ admits an exhaustion plurisubharmonic
function of logarithmic growth which is strictly $\left( n-k\right) $%
-convex. Then, if $\dim X\geqslant 3$, he proved that there are no Levi flat
real analytic hypersurfaces such that the normal bundle to the Levi
foliation admits a fiber metric whose curvature is semipositive of rank$%
\geqslant 2$ on the holomorphic tangent space to $L$.$\ $Some possibilities
for generalization in the smooth case are also indicated.

In this paper we solve the above mentioned conjecture of Brunella for
compact connected K\"{a}hler manifolds of dimension $n\geqslant 3$. The
principal ingredient of the proof is a refinement of the proof of Brunella 
\cite{Brunella08} of the strong pseudoconvexity of $X\backslash L$ : we show
that there exist a neighborhood $U$ of $L$ and a function $v$ on $U$
vanishing on $L$, such that $-i\partial \overline{\partial }\ln v\geqslant
c\omega $ on $U\backslash L$, where $c>0$ and $\omega $ is the $\left(
1,1\right) $-form associated to the K\"{a}hler metric. Then we use the $L^{2}
$ estimates \cite{Andreotti61}, \cite{AV65}, \cite{HO65}, \cite{DE82} for
the weighted $\overline{\partial }$-equation on $\left( n,q\right) $-forms
on $X\backslash L$ endowed with a complete K\"{a}hler metric. These
estimates together with the lower uniform boundedness of the eigenvalues of
the Levi form and a duality method developped in \cite{Henkin00}, allow us
to solve the $\overline{\partial }$-equation with compact support for $%
\left( 0,q\right) $-forms, $1\leqslant q\leqslant n-1$, and this leads in
dimensions $\geqslant 3$ to the solution of Brunella's conjecture.

\section{Preliminaries}

Let $X$ be a complex $n$-dimensional manifold, $\omega $ a K\"{a}hler metric
on $X$, $\Omega $ a domain in $X$ and $\sigma $ a positive function on $%
\Omega $. For $\alpha \in \mathbb{R}$ denote 
\begin{equation*}
L_{(p,q)}^{2}(\Omega ,\sigma ^{\alpha },\omega )=\left\{ f\in L_{\left(
p,q\right) loc}^{2}\left( \Omega \right) :\int_{\Omega }\left\vert
f\right\vert ^{2}\sigma ^{2\alpha }dV_{\omega }<\infty \right\}
\end{equation*}%
endowed with the norm%
\begin{equation*}
N_{\alpha ,\omega ,\sigma }(f)=\left( \int_{\Omega }\left\vert f\right\vert
^{2}\sigma ^{2\alpha }dV_{\omega }\right) ^{1/2}.
\end{equation*}

Let $\Omega $ be a pseudoconvex domain in $\mathbb{CP}_{n}$ and $\delta
_{\partial \Omega }$ the geodesic distance to the boundary for the
Fubini-Study metric $\omega _{FS}$. By using the $L^{2}$ estimates for the $%
\overline{\partial }$-operator of H\"{o}rmander with the weight $e^{-\varphi
}$, $\varphi =-\alpha \log \delta _{\partial \Omega }$ which is strongly
plurisubharmonic by a theorem of Takeuchi \cite{Takeushi67}, Henkin and
Iordan proved in \cite{Henkin00} the existence and regularity of the $%
\overline{\partial }$ equation for $\overline{\partial }$-closed forms in $%
L_{(p,q)}^{2}(\Omega ,\delta _{\partial \Omega }^{-\alpha },\omega _{FS})$
verifying the moment condition. This gives the regularity of the $\overline{%
\partial }$-operator in pseudoconcave domains with Lipschitz boundary \cite%
{Henkin00} and, by using a method of Siu \cite{Siu00}, \cite{Siu02}, the
nonexistence of smooth Levi flat hypersurfaces in $\mathbb{CP}_{n}$, $%
n\geqslant 3$ follows (see \cite{Iordan08}). These techniques will be used
in the 4th and the 5th paragraph.

We will use also the following theorem of regularity of $\overline{\partial }
$ equation of Brinkschulte \cite{Brinkschulte2004}:

\begin{theorem}
\label{Brinkschulte} Let $\Omega $ be a relatively compact domain with
Lipschitz boundary in a K\"{a}hler manifold $\left( X,\omega \right) $ and
set $\delta _{\partial \Omega }$ the geodesic distance to the boundary of $%
\Omega $. Let $f\in L_{(p,q)}^{2}(\Omega ,\delta _{\partial \Omega
}^{-k},\omega )\cap C_{(p,q)}^{k}\left( \overline{\Omega }\right) \cap
C_{(p,q)}^{\infty }\left( \Omega \right) $, $q\geqslant 1$, $k\in \mathbb{N}$
and $u\in L_{(p,q-1)}^{2}(\Omega ,\delta _{\partial \Omega }^{-k},\omega )$
such that $\overline{\partial }u=f$ and $\overline{\partial }_{-k}^{\ast
}u=0 $, where $\overline{\partial }_{-k}^{\ast }$ is the Hilbert space
adjoint of the unbounded operator $\overline{\partial }%
_{-k}:L_{(p,q-1)}^{2}(\Omega ,\delta _{\partial \Omega }^{-k},\omega
)\rightarrow L_{(p,q)}^{2}(\Omega ,\delta _{\partial \Omega }^{-k},\omega )$%
. Then for $k$ big enough $u\in C_{(p,q-1)}^{s\left( k\right) }\left( 
\overline{\Omega }\right) $ where $s\left( k\right) \underset{k\rightarrow
\infty }{\thicksim }\sqrt{k}$.
\end{theorem}

\section{Strong pseudoconvexity of the complement of a Levi flat hypersurface%
}

Let $L$ be a smooth Levi flat hypersurface in a Hermitian manifold $X$. As
was mentioned in \cite{Brunella08} and in \cite{Ohsawa2013}, by taking a
double covering, we can assume that $L$ is orientable and the complement of $%
L$ has two connected components in a neighborhood of $L$. This will be
always supposed in the sequel and for an open neighborhood $U$ of $L$ we
will denote by $U^{+}$ and $U^{-}$ the two connected components of $%
U\backslash L$. We will denote by $\delta _{L}$ the signed geodesic distance
to $L$.

In \cite{Brunella08} Brunella proved that the complement of a closed Levi
flat hypersurface in a compact Hermitian manifold of class $C^{2,\alpha }$, $%
0<\alpha <1$, having the property that the Levi foliation extends to a
holomorphic foliation in a neighborhood of $L$ and the normal bundle to the
Levi foliation admits a $C^{2}$ Hermitian metric with leafwise positive
curvature is strongly pseudoconvex, i.e. there exists an exhaustion function
which is strongly plurisubharmonic outside a compact set. The following
proposition strenghtens this result:

\begin{proposition}
\label{Strongly pseudoconvex}Let $L$ be a compact $C^{3}$ Levi flat
hypersurface in a Hermitian manifold $X$ of dimension $n\geqslant 2$, such
that the normal bundle $\mathcal{N}_{L}^{1,0}$ to the Levi foliation admits
a $C^{2}$ Hermitian metric with leafwise positive curvature. Then there
exist a neighborhood $U$ of $L$, $c>0$ and a non-negative function $v\in
C^{2}\left( U\right) $, vanishing on $L$ and positive on $U\backslash L$
such that $-i\partial \overline{\partial }\ln v\geqslant c\omega $ on $%
U\backslash L$, where $\omega $ is the $\left( 1,1\right) $-form associated
to the metric. Moreover, there exists a nonvanishing continuous function $g$
in a neighborhood of $L$ such that $v=g\delta _{L}^{2}$.
\end{proposition}

\begin{proof}
Let $z_{0}\in L$. There exist holomorphic coordinates $z=\left( z_{1},\cdot
\cdot \cdot ,z_{n-1},z_{n}\right) =\left( z^{\prime },z_{n}\right) $ in a
neighborhood of $z_{0}$ such that the local parametric equations for $L$ are
of the form%
\begin{equation*}
z_{j}=w_{j},\ j=1,...,n-1,\ z_{n}=\varphi \left( w^{\prime },t\right)
\end{equation*}%
where $\varphi $ is of class $C^{3}$ (see \cite{Barrett88}) on a
neighborhood of the origin in $\mathbb{C}^{n-1}\times \mathbb{R}$,
holomorphic in $w^{\prime }$ and $\frac{\partial \varphi }{\partial t}\left(
z_{0}\right) \in \mathbb{R}^{\ast }$. We consider a $C^{3}$ extension $\psi
=\left( \psi _{1},...,\psi _{n}\right) $ of $\varphi $ on a neighborhood of
the origin in $\mathbb{C}^{n-1}\times \mathbb{C}$, $\psi \left( w^{\prime
},t+is\right) =\left( w^{\prime },\varphi \left( w^{\prime },t\right)
+is\right) $. Then $\psi $ is a $C^{3}$ diffeomorphism in a neighborhood of $%
z_{0}$ and holomorphic in $w^{\prime }$. It follows that 
\begin{equation*}
L=\left\{ \left( z^{\prime },z_{n}\right) :\rho \left( z^{\prime
},z_{n}\right) =0\right\} .
\end{equation*}%
where$\ \rho =\func{Im}\left( \psi ^{-1}\right) _{n}$. We denote $f=\left(
\psi ^{-1}\right) _{n}\left( z^{\prime },z_{n}\right) $. Since $\overline{%
\partial }_{b}f=0$ on $L$, where $\overline{\partial }_{b}$ is the
tangential Cauchy-Riemann operator on $L$, there exists an extension $%
\widetilde{f}$ of class $C^{3}$ in a neighborhood of $z_{0}$ such that $%
\overline{\partial }\widetilde{f}$ vanishes to order greater than $2$ on $L$%
, i.e. $D^{l}\overline{\partial }\widetilde{f}=0$ for $\left\vert
l\right\vert \leqslant 2$ on $L$.

So there exists an open finite covering $\left( \widetilde{U_{j}}\right)
_{j\in J}$ by holomorphic charts of $L$ such that $\widetilde{U_{j}}%
\backslash L=\widetilde{U_{j}^{+}}\cup \widetilde{U_{j}^{-}}$ such that $%
U_{j}=L\cap \widetilde{U_{j}}=\left\{ z\in \widetilde{U_{j}}:\func{Im}%
\widetilde{f_{j}}=0\right\} $, where $\overline{\partial }\widetilde{f_{j}}$
vanishes to order greater than $2$ on $L$ and the Levi foliation is given on 
$U_{j}$ by $\left\{ z\in U_{j}:\ \widetilde{f_{j}}\left( z\right)
=c_{j}\right\} $, $c_{j}\in \mathbb{R}$ . Thus $d\widetilde{f_{j}}=\partial 
\widetilde{f_{j}}$ is a nonvanishing section of $\mathcal{N}_{L}^{1,0}$ on $%
U_{j}$ and by shrinking $\widetilde{U_{j}}$, we may consider that $d%
\widetilde{f_{j}}\neq 0$ on $\widetilde{U_{j}}$.

We may suppose that $\mathcal{N}_{L}^{1,0}$ is represented by a cocycle $%
\left\{ g_{jk}\right\} $ of class $C^{2}$ subordinated to the covering $%
\left( U_{j}\right) _{j\in J}$ and there exist closed $\left( 1,0\right) $%
-forms $\alpha _{j}$ of class $C^{2}$ on $U_{j}$ holomorphic along the
leaves such that $T^{1,0}\left( U_{j}\right) =\ker \alpha _{j}$ for every $%
j\in J$ and $\alpha _{j}=g_{jk}\alpha _{k}$ on $U_{j}\cap U_{k}$. So $\left(
\alpha _{j}\right) _{j\in J}$ defines a global form $\alpha $ on $L$ with
values in $\mathcal{N}_{L}^{1,0}$ such that locally on $U_{j}$ we have $%
\alpha \left( z\right) =\alpha _{j}\left( z\right) \otimes \alpha _{j}^{\ast
}\left( z\right) $ where $\alpha _{j}^{\ast }$ is the dual frame of $\alpha
_{j}$. In particular we have $\alpha _{k}^{\ast }=g_{jk}\alpha _{j}^{\ast }$.

Let $h$ be a $C^{2}$ Hermitian metric with positive leafwise curvature $%
\Theta _{h}\left( \mathcal{N}_{L}^{1,0}\right) $ on $\mathcal{N}_{L}^{1,0}$. 
$h$ is defined on each $U_{j}$ by a $C^{2}$ function $h_{j}=$ $\left\vert
\alpha _{j}^{\ast }\right\vert ^{2}$ such that $h_{k}=\left\vert
g_{jk}\right\vert ^{2}h_{j}$ on $U_{j}\cap U_{k}$.

Since $\alpha _{j}=\eta _{j}d\widetilde{f_{j}}$ on $U_{j}$ for every $j$,
where $\eta _{j}$ are nowhere vanishing functions of class $C^{2}$ on $U_{j}$
holomorphic along the leaves and 
\begin{equation*}
\frac{1}{\eta _{k}}\left( d\widetilde{f_{k}}\right) ^{\ast }=\frac{1}{\eta
_{j}}g_{jk}\left( d\widetilde{f_{j}}\right) ^{\ast }
\end{equation*}%
on $U_{j}\cap U_{k}$, it follows that 
\begin{equation*}
\left\vert g_{jk}\left( z\right) \right\vert ^{2}=\left\vert \frac{\eta
_{j}\left( z\right) }{\eta _{k}\left( z\right) }\right\vert ^{2}\left\vert 
\frac{\left( d\widetilde{f_{k}}\right) ^{\ast }}{\left( d\widetilde{f_{j}}%
\right) ^{\ast }}\right\vert ^{2}=\frac{h_{k}\left( z\right) }{h_{j}\left(
z\right) },\ z\in U_{j}\cap U_{k}.
\end{equation*}

So 
\begin{equation*}
h_{j}\left\vert \eta _{j}\right\vert ^{2}\left( \func{Im}\widetilde{f_{j}}%
\right) ^{2}-h_{k}\left\vert \eta _{k}\right\vert ^{2}\left( \func{Im}%
\widetilde{f_{k}}\right) ^{2}
\end{equation*}%
vanishes to order greater than $2$ on $U_{j}\cap U_{k}$ and $\left(
h_{j}\left\vert \eta _{j}\right\vert ^{2}\left( \func{Im}\widetilde{f_{j}}%
\right) ^{2}\right) _{j\in J}$ defines a jet of order $2$ on $L$. By Whitney
extension theorem there exists a $C^{2}$ function $v$ on $X$ such that $%
v-h_{j}\left\vert \eta _{j}\right\vert ^{2}\left( \func{Im}\widetilde{f_{j}}%
\right) ^{2}$ vanishes to order $2$ on $U_{j}$ for every $j\in J$. Let $%
\widetilde{\eta _{j}},\widetilde{h_{j}}$ be $C^{2}$ extensions of $\eta
_{j},h_{j}$ on $\widetilde{U_{j}}$ and set $\widetilde{\alpha _{j}}=%
\widetilde{\eta _{j}}d\widetilde{f_{j}}$, $\widetilde{v}=\widetilde{h_{j}}%
\left\vert \widetilde{\eta _{j}}\right\vert ^{2}\left( \func{Im}\widetilde{%
f_{j}}\right) ^{2}$.

For $z\in \widetilde{U_{j}}$ denote $E_{z}^{\prime }=\left\{ V^{\prime }\in
T_{z}^{1,0}\left( X\right) :\ \left\langle \partial \func{Im}\widetilde{f_{j}%
},V^{\prime }\right\rangle =0\right\} $ and $E_{z}^{^{\prime \prime }}$ the
orthogonal of $E_{z}^{\prime }$ in $T_{z}^{1,0}\left( X\right) $. Then for
every $V\in T_{z}^{1,0}\left( X\right) $ there exists $V^{\prime }\in
E_{z}^{\prime },V^{\prime \prime }\in E_{z}^{\prime \prime }$ such that $%
V=V^{\prime }+V^{\prime \prime }$. The curvature form $\Theta \left( 
\mathcal{N}_{L}^{1,0}\right) $ is represented by $-i\partial \overline{%
\partial }\ln \left( h_{j}\left\vert \alpha _{j}\right\vert ^{2}\right) $ on 
$U_{j}$, so by shrinking $\widetilde{U_{j}}$ we may suppose that there
exists $\beta >0$ such that $\left( -i\partial \overline{\partial }\ln
\left( \widetilde{h_{j}}\left\vert \widetilde{\alpha _{j}}\right\vert
^{2}\right) \right) \left( V^{\prime },\overline{V^{\prime }}\right)
\geqslant \beta \omega \left( V^{\prime },\overline{V^{\prime }}\right) $
for every $z\in \widetilde{U_{j}}$ and $V\in T_{z}^{1,0}\left( X\right) $.

On $\widetilde{U_{j}}\backslash L$ we have 
\begin{eqnarray}
-i\partial \overline{\partial }\ln \widetilde{v} &=&-i\partial \overline{%
\partial }\ln \left( \widetilde{h_{j}}\left\vert \frac{\widetilde{\alpha _{j}%
}}{d\widetilde{f_{j}}}\right\vert ^{2}\left( \func{Im}\widetilde{f_{j}}%
\right) ^{2}\right)  \label{A} \\
&=&-i\partial \overline{\partial }\ln \widetilde{h_{j}}\left\vert \widetilde{%
\alpha _{j}}\right\vert ^{2}+i\partial \overline{\partial }\ln \left\vert d%
\widetilde{f_{j}}\right\vert ^{2}-i\partial \overline{\partial }\ln \left( 
\func{Im}\widetilde{f_{j}}\right) ^{2}.  \notag
\end{eqnarray}%
Let $z\in \widetilde{U_{j}}$ and $V\in T_{z}^{1,0}\left( X\right) $. Then $%
V=V^{\prime }+V^{\prime \prime }$, $V^{\prime }\in E_{z}^{\prime }$ and $%
V"\in E_{z}^{\prime \prime }$ and%
\begin{eqnarray*}
-i\partial \overline{\partial }\ln \widetilde{h_{j}}\left\vert \widetilde{%
\alpha _{j}}\right\vert ^{2}\left( V,\overline{V}\right) &=&\left(
-i\partial \overline{\partial }\ln \left( \widetilde{h_{j}}\left\vert 
\widetilde{\alpha _{j}}\right\vert ^{2}\right) \right) \left( V^{\prime },%
\overline{V^{\prime }}\right) \\
&&+2\func{Re}\left( -i\partial \overline{\partial }\ln \left( \widetilde{%
h_{j}}\left\vert \widetilde{\alpha _{j}}\right\vert ^{2}\right) \left(
V^{\prime },\overline{V^{\prime \prime }}\right) \right) \\
&&+\left( -i\partial \overline{\partial }\ln \left( \widetilde{h_{j}}%
\left\vert \widetilde{\alpha _{j}}\right\vert ^{2}\right) \right) \left(
V^{\prime \prime },\overline{V^{\prime \prime }}\right)
\end{eqnarray*}%
There exists a constant $C>0$ depending on the eigenvalues of $-i\partial 
\overline{\partial }\ln \left( \widetilde{h_{j}}\left\vert \widetilde{\alpha
_{j}}\right\vert ^{2}\right) $ with respect to $\omega $ such that for every 
$\varepsilon >0$%
\begin{equation*}
2\left\vert \func{Re}\left( -i\partial \overline{\partial }\ln \left( 
\widetilde{h_{j}}\left\vert \widetilde{\alpha _{j}}\right\vert ^{2}\right)
\left( V^{\prime },\overline{V^{\prime \prime }}\right) \right) \right\vert
\leqslant C\left( \varepsilon \omega \left( V^{\prime },\overline{V^{\prime }%
}\right) +\frac{1}{\varepsilon }\omega \left( V^{\prime \prime },\overline{%
V^{\prime \prime }}\right) \right) ,
\end{equation*}%
so%
\begin{eqnarray}
-i\partial \overline{\partial }\ln \widetilde{h_{j}}\left\vert \widetilde{%
\alpha _{j}}\right\vert ^{2}\left( V,\overline{V}\right) &\geqslant &\beta
\omega \left( V^{\prime },\overline{V^{\prime }}\right) -C\left( \varepsilon
\omega \left( V^{\prime },\overline{V^{\prime }}\right) -\frac{1}{%
\varepsilon }\omega \left( V^{\prime \prime },\overline{V^{\prime \prime }}%
\right) \right)  \notag \\
&&-\left\Vert -i\partial \overline{\partial }\ln \widetilde{h_{j}}\left\vert 
\widetilde{\alpha _{j}}\right\vert ^{2}\right\Vert _{\omega }\omega \left(
V^{\prime \prime },\overline{V^{\prime \prime }}\right)  \label{B}
\end{eqnarray}

Since $\overline{\partial }\widetilde{f_{j}}$ vanishes to order greater than 
$2$ on $L$, for every $\gamma >0$ there exists a neighborhood of $L$ such
that 
\begin{equation}
\left\vert i\partial \overline{\partial }\ln \left\vert d\widetilde{f_{j}}%
\right\vert ^{2}\left( V,\overline{V}\right) \right\vert \leqslant \gamma
\omega \left( V,\overline{V}\right)  \label{C}
\end{equation}%
and 
\begin{equation}
\left\vert i\partial \overline{\partial }\func{Im}\widetilde{f_{j}}\left( V,%
\overline{V}\right) \right\vert \leqslant \gamma \left( \func{Im}\widetilde{%
f_{j}}\right) \omega \left( V,\overline{V}\right) .  \label{D}
\end{equation}%
Let $z\in \widetilde{U_{j}}\backslash L$. By (\ref{D}) it follows that%
\begin{eqnarray}
-i\partial \overline{\partial }\ln \left( \func{Im}\widetilde{f_{j}}\right)
^{2}\left( V,\overline{V}\right) &=&\left( -2\frac{i\partial \overline{%
\partial }\func{Im}\widetilde{f_{j}}}{\func{Im}\widetilde{f_{j}}}+2i\frac{%
\partial \func{Im}\widetilde{f_{j}}\wedge \overline{\partial }\func{Im}%
\widetilde{f_{j}}}{\left( \func{Im}\widetilde{f_{j}}\right) ^{2}}\right)
\left( V,\overline{V}\right)  \notag \\
&\geqslant &-2\gamma \omega \left( V,\overline{V}\right) +2i\frac{\partial 
\func{Im}\widetilde{f_{j}}\wedge \overline{\partial }\func{Im}\widetilde{%
f_{j}}}{\left( \func{Im}\widetilde{f_{j}}\right) ^{2}}\left( V^{\prime
\prime },\overline{V^{\prime \prime }}\right)  \label{E} \\
&\geqslant &-2\gamma \omega \left( V,\overline{V}\right) +\frac{2\underset{%
\widetilde{U_{j}}}{\inf }\left\Vert \partial \func{Im}\widetilde{f_{j}}%
\right\Vert _{\omega }^{2}}{\left( \func{Im}\widetilde{f_{j}}\right) ^{2}}%
\omega \left( V^{\prime \prime },\overline{V^{\prime \prime }}\right) . 
\notag
\end{eqnarray}%
By using (\ref{B}),\ (\ref{C}) and (\ref{E}), from (\ref{A}) we obtain%
\begin{eqnarray*}
\left( -i\partial \overline{\partial }\ln \widetilde{v}\right) \left( V,%
\overline{V}\right) &\geqslant &\left( \beta -C\varepsilon \right) \omega
\left( V^{\prime },\overline{V^{\prime }}\right) \\
&&+\left( \frac{2}{\left( \func{Im}f_{j}\right) ^{2}}\underset{\widetilde{%
U_{j}}}{\inf }\left\Vert \partial \func{Im}\widetilde{f_{j}}\right\Vert
_{\omega }^{2}-\frac{C}{\varepsilon }-\left\Vert -i\partial \overline{%
\partial }\ln \widetilde{h_{j}}\left\vert \widetilde{\alpha _{j}}\right\vert
^{2}\right\Vert _{\omega }\right) \omega \left( V^{\prime \prime },\overline{%
V^{\prime \prime }}\right) \\
&&-2\gamma \omega \left( V,\overline{V}\right) .
\end{eqnarray*}

By choosing $0<C\varepsilon <\beta $ and by shrinking $\widetilde{U_{j}}$
such that $\frac{2}{\left( \func{Im}f_{j}\right) ^{2}}$ is big enough and $%
\gamma $ small enough, we obtain that there exists $c>0$ such that $%
-i\partial \overline{\partial }\ln \widetilde{v}\geqslant c\omega $ on $%
\widetilde{U_{j}}\backslash L$. Finally, since $v-\widetilde{v}$ vanishes to
order greater than $2$ on $L$, it follows that there exists a neighborhood $%
U^{\prime }$ of $L$ such that $-\ln v$ is strongly plurisubharmonic on $%
U^{\prime }\backslash L$. We can now take $U=\left\{ z\in U^{\prime }:\
v\left( z\right) <\mu \right\} $ for $\mu >0$ small enough.

$L$ is a $C^{3}$ manifold, so the signed distance function $\delta_L $ is a
defining function of class $C^{3}$ for $L$. Since $v$ is of class $C^{2}$ on 
$U$ and vanishes to order greater than $2$ on $L$, we have $v=g\delta_L ^{2}$
with $g$ continuous in a neighborhood of $L$.

Suppose that there exists $x\in L$ such that $g\left( x\right) =0$. Then $%
v=o\left( \delta _{L}^{2}\right) $ in a neighborhood of $x$. But there
exists $j$ such that $x\in U_{j}$ and $v=h_{j}\left\vert \eta
_{j}\right\vert ^{2}\left( \func{Im}\widetilde{f_{j}}\right) ^{2}+o\left(
\delta _{L}^{2}\right) $. Since $\func{Im}\widetilde{f_{j}}=0$ and $d\func{Im%
}\widetilde{f_{j}}\neq 0$ on $L$ it follows that $\left\vert \nabla
^{2}v\right\vert \left( x\right) \neq 0$. This contradiction shows that $%
g\left( x\right) \neq 0$ on $L$.
\end{proof}

\section{Weighted estimates for the $\overline{\partial }$-equation}

\begin{remark}
\label{Remark}Under the hypothesis and conclusions of Proposition \ref%
{Strongly pseudoconvex}, we consider a positive extension $\widetilde{v}$ of
the restriction of $v$ on a neighborhood of $L$ to $X\backslash L$. Let $s>0$
such that $\left\{ v<e^{-s}\right\} \subset U$ and let $\varphi $ be a
smooth function on $\mathbb{R}$ such that $\varphi =0$ on $]-\infty ,s]$ and 
$\varphi $ is strictly convex increasing on $]s,\infty \lbrack $. Then $\psi
=\varphi \left( -\ln \widetilde{v}\right) $ is a plurisubharmonic exhaustion
function of $X\backslash L$, which is strongly plurisubharmonic outside a
compact subset of $X\backslash L$.

In the sequel, $L$ will be a compact $C^{\infty }$ Levi flat hypersurface in
a compact K\"{a}hler manifold $X$ of dimension $n\geqslant 2$, verifying the
hypothesis and the conclusions of Proposition \ref{Strongly pseudoconvex}.
We denote $X^{\pm }$ the connected components of $\left\{ z\in X:\
v>0\right\} $ endowed with a complete K\"{a}hler metric $\widetilde{\omega }$
which will be defined later and we set 
\begin{equation*}
\mathcal{D}_{(p,q)}(X^{\pm })=\left\{ f\in C_{\left( p,q\right) }^{\infty
}\left( X^{\pm }\right) :\ supp~f\subset \subset X^{\pm }\right\}
\end{equation*}%
and 
\begin{equation*}
\mathcal{H}_{\left( p,q\right) }\left( X^{\pm },\widetilde{v}^{\alpha },%
\widetilde{\omega }\right) =\ker \overline{\partial }\cap \ker \overline{%
\partial }_{\alpha }^{\ast }\subset L_{\left( p,q\right) }^{2}\left( X^{\pm
},\widetilde{v}^{\alpha },\widetilde{\omega }\right)
\end{equation*}%
where $\overline{\partial }_{\alpha }^{\ast }$ is the Hilbert space adjoint
of the operator $\overline{\partial }:L_{\left( p,q\right) }^{2}\left(
X^{\pm },\widetilde{v}^{\alpha },\widetilde{\omega }\right) \rightarrow
L_{\left( p,q+1\right) }^{2}\left( X^{\pm },\widetilde{v}^{\alpha },%
\widetilde{\omega }\right) $.
\end{remark}

\begin{proposition}
\label{Solution d bar k>>0}For every $\alpha >0$, there exists a complete K%
\"{a}hler metric $\widetilde{\omega }$ on $X\backslash L$, $\omega \leqslant 
\widetilde{\omega }\leqslant \frac{C}{\widetilde{v}^{2}}\omega $, $C>0$,
such that the range $\mathcal{R}_{\left( n,q\right) }^{\alpha }\left( X^{\pm
}\right) $ of the operator $\overline{\partial }_{\alpha }:L_{\left(
n,q-1\right) }^{2}\left( X^{\pm },\widetilde{v}^{\alpha },\widetilde{\omega }%
\right) \rightarrow L_{\left( n,q\right) }^{2}\left( X^{\pm },\widetilde{v}%
^{\alpha },\widetilde{\omega }\right) $ is closed for $1\leqslant q\leqslant
n$.
\end{proposition}

\begin{proof}
The proof is based on methods of \cite{Demailly85} (see also \cite%
{Demaillylivre}).

Denote by $\omega $ the K\"{a}hler metric of $X$. Since $i\partial \overline{%
\partial }\left( -\ln v\right) \geqslant c\omega $ on $U\backslash L$, $c>0$%
, by a method developped in \cite{OS98} it follows that there exist a
neigborhood $V$ of $L$ and $\eta >0$ such that $-v^{\eta }$ is strongly
plurisubharmonic on $V\backslash L$. Then for $0<\beta <\eta $, we have the
Donnelly-Fefferman estimate \cite{Donnelly83} 
\begin{equation}
i\partial \left( -\ln v\right) \wedge \overline{\partial }\left( -\ln
v\right) \leqslant ir\partial \overline{\partial }\left( -\ln v\right) .
\label{Donnelly}
\end{equation}%
on $V\backslash L$, with $0<r=\beta /\eta <1$. This is equivalent to say
that the norm of $\partial \left( -\ln v\right) $ measured in the metric $%
i\partial \overline{\partial }\left( -\ln v\right) $ is smaller than $r$ on $%
V\backslash L$ (see also \cite{Berndtsson00} and \cite{Henkin00}).

Let $\alpha >0$. We consider the trivial line bundle $E$ on $X\backslash L$
endowed with the Hermitian metric $h_{\alpha }=e^{\alpha \ln \widetilde{v}}$%
. Set 
\begin{equation*}
\widetilde{\omega }=i\Theta \left( E\right) +K\omega =i\alpha \partial 
\overline{\partial }\left( -\ln \widetilde{v}\right) +K\omega
\end{equation*}%
with $K$ a positive constant. Since $-\ln \widetilde{v}$ is an exhaustion
function on $X\backslash L$, it follows by (\ref{Donnelly}) that for $K$ big
enough $\widetilde{\omega }$ is a complete K\"{a}hler metric on $X\backslash
L$ such that $\omega \leqslant \widetilde{\omega }\leqslant \frac{C}{%
\widetilde{v}^{2}}\omega $, $C>0$.

Denote $\lambda _{j}$ (respectively $\widetilde{\lambda }_{j}$) the
eigenvalues of $i\Theta \left( E\right) $ with respect to $\omega $
(respectively $\widetilde{\omega }$), $1\leqslant j\leqslant n$, in
increasing order. By Proposition \ref{Strongly pseudoconvex}, there exists $%
c>0$ such that $i\Theta \left( E\right) =i\alpha \partial \overline{\partial 
}\left( -\ln \widetilde{v}\right) \geqslant \alpha c\omega $ on $\left\{
\psi >b\right\} $ for $b$ big enough. So, as in \cite{Demailly85} (1.6) we
have%
\begin{equation}
1\geqslant \widetilde{\lambda }_{j}=\frac{\lambda _{j}}{\lambda _{j}+K}%
\geqslant \frac{\alpha c}{\alpha c+K}>0,\ 1\leqslant j\leqslant n
\label{lambda}
\end{equation}%
on $\left\{ \psi >b\right\} $. By Bochner-Kodaira-Nakano inequality (see for
ex. \cite{Demaillylivre}) we have 
\begin{equation}
N_{\alpha ,\widetilde{\omega },\widetilde{v}}\left( \overline{\partial }%
u\right) ^{2}+N_{\alpha ,\widetilde{\omega },\widetilde{v}}\left( \overline{%
\partial }_{\alpha }^{\ast }u\right) ^{2}\geqslant \int_{X^{\pm
}}\left\langle \left( \left[ i\Theta \left( E\right) ,\Lambda _{\widetilde{%
\omega }}\right] \right) u,u\right\rangle _{\alpha ,\widetilde{\omega },%
\widetilde{v}}dV_{\widetilde{\omega }}  \label{BKN}
\end{equation}%
for every $u\in \mathcal{D}_{\left( n,q\right) }\left( X\backslash L\right) $%
, where $N_{\alpha ,\widetilde{\omega },\widetilde{v}}=\int_{X^{\pm
}}\left\vert u\right\vert _{\widetilde{\omega }}^{2}\widetilde{v}^{\alpha
}dV_{\widetilde{\omega }}$.

Let $\chi $ be a smooth function on $X$ such that $0\leqslant \chi \leqslant
1$, $\chi =0$ on a neighborhood of $\left\{ \psi <b\right\} $ and $\chi =1$
on a neighborhood $\left\{ \psi >b^{\prime }\right\} $ of $L$, $b^{\prime
}>b $. By (\ref{BKN}) and (\ref{lambda}), for every $u\in \mathcal{D}%
_{\left( n,q\right) }\left( X\backslash L\right) $ we have%
\begin{eqnarray*}
N_{\alpha ,\widetilde{\omega },\widetilde{v}}\left( \overline{\partial }%
\left( \chi u\right) \right) ^{2}+N_{\alpha ,\widetilde{\omega },\widetilde{v%
}}\left( \overline{\partial }_{\alpha }^{\ast }\left( \chi u\right) \right)
^{2} &\geqslant &\int_{X^{\pm }}\left\langle \left( \left[ i\Theta \left(
E\right) ,\Lambda _{\widetilde{\omega }}\right] \right) \chi u,\chi
u\right\rangle _{\alpha ,\widetilde{\omega },\widetilde{v}}dV_{\widetilde{%
\omega }} \\
&\geqslant &\int_{\left\{ \psi >b^{\prime }\right\} }\left\langle \left( 
\left[ i\Theta \left( E\right) ,\Lambda _{\widetilde{\omega }}\right]
\right) \chi u,\chi u\right\rangle _{\alpha ,\widetilde{\omega },\widetilde{v%
}}dV_{\widetilde{\omega }} \\
&\geqslant &\int_{\left\{ \psi >b^{\prime }\right\} }\left( \lambda
_{1}+\cdot \cdot \cdot +\lambda _{n}\right) \left\vert \chi u\right\vert _{%
\widetilde{\omega }}^{2}\widetilde{v}^{\alpha }dV_{\widetilde{\omega }} \\
&\geqslant &\frac{\alpha c}{\alpha c+K}\int_{\left\{ \psi >b^{\prime
}\right\} }\left\vert u\right\vert _{\widetilde{\omega }}^{2}\widetilde{v}%
^{\alpha }dV_{\widetilde{\omega }}
\end{eqnarray*}%
so there exists $C,c^{\prime }>0$ such that%
\begin{eqnarray*}
&&2N_{\alpha ,\widetilde{\omega },\widetilde{v}}\left( \overline{\partial }%
u\right) ^{2}+2N_{\alpha ,\widetilde{\omega },\widetilde{v}}\left( \overline{%
\partial }_{\alpha }^{\ast }u\right) ^{2}+C\int_{supp\left( \chi ^{\prime
}\right) }\left\vert u\right\vert _{\widetilde{\omega }}^{2}\widetilde{v}%
^{\alpha }dV_{\widetilde{\omega }} \\
&\geqslant &c^{\prime }\int_{X\backslash L}\left\vert u\right\vert _{%
\widetilde{\omega }}^{2}\widetilde{v}^{\alpha }dV_{\widetilde{\omega }%
}-c^{\prime }\int_{\left\{ \psi <b^{\prime }\right\} }\left\vert
u\right\vert _{\widetilde{\omega }}^{2}\widetilde{v}^{\alpha }dV_{\widetilde{%
\omega }}.
\end{eqnarray*}%
Finally it follows that there exists a compact subset $F=supp\left( \chi
^{\prime }\right) \cup \left\{ \psi \leqslant b^{\prime }\right\} $ of $%
X^{\pm }$ such that for every $u\in \mathcal{D}_{\left( n,q\right) }\left(
X\backslash L\right) $%
\begin{equation}
c^{\prime }N_{\alpha ,\widetilde{\omega },\widetilde{v}}\left( u\right)
^{2}\leqslant 2N_{\alpha ,\widetilde{\omega },\widetilde{v}}\left( \overline{%
\partial }u\right) ^{2}+2N_{\alpha ,\widetilde{\omega },\widetilde{v}}\left( 
\overline{\partial }_{\alpha }^{\ast }u\right) ^{2}+\left( C+c^{\prime
}\right) \int_{F}\left\vert u\right\vert _{\widetilde{\omega }}^{2}%
\widetilde{v}^{\alpha }dV_{\widetilde{\omega }}.  \label{fund}
\end{equation}%
Since $\widetilde{\omega }$ is a complete metric on $X\backslash L$, (\ref%
{fund}) is valid for every $u\in \left( Dom\overline{\partial }\right) \cap
\left( Dom\overline{\partial }_{\alpha }^{\ast }\right) $. The conclusion of
Proposition \ref{Solution d bar k>>0} is now a consequence of Proposition
1.2 of \cite{Ohsawa82}.
\end{proof}

\begin{corollary}
\label{Harmonis space =0}For every $\alpha >0$ and $1\leqslant q\leqslant n$
we have $\mathcal{H}_{\left( n,q\right) }\left( X^{\pm },\widetilde{v}%
^{\alpha },\widetilde{\omega }\right) =\left\{ 0\right\} $.
\end{corollary}

\begin{proof}
As $\left( X^{\pm },\widetilde{\omega }\right) $ is a connected weakly $1$%
-complete K\"{a}hler manifold and the bundle $E$ defined in the proof of
Theorem \ref{Solution d bar k>>0} is a semi-positive line bundle on $X^{\pm
} $ which is positive outside a compact suset of $X^{\pm }$, the Corollary %
\ref{Harmonis space =0} is a consequence of \cite{Takegoshi81}, Corollary of
the Main Theorem (see also \cite{Aronszain57}, \cite{Riemenschneider71} and 
\cite{Ohsawabook}, Corollary 2.10).
\end{proof}

By taking in account Corollary \ref{Harmonis space =0}, a classical
application of Proposition \ref{Solution d bar k>>0} (see for example \cite%
{Folland72}) is the following:

\begin{corollary}
\label{Neumann operator}For every $\alpha >0$ and , $1\leqslant q\leqslant n$
we have:

\begin{enumerate}
\item There exists the $\overline{\partial }$-Neumann operator $\mathcal{N}%
_{\left( n,q\right) }^{\alpha }:L_{\left( n,q\right) }^{2}\left( X^{\pm },%
\widetilde{v}^{\alpha },\widetilde{\omega }\right) \rightarrow L_{\left(
n,q\right) }^{2}\left( X^{\pm },\widetilde{v}^{\alpha },\widetilde{\omega }%
\right) $ such that for every $f\in L_{\left( n,q\right) }^{2}\left( X^{\pm
},\widetilde{v}^{\alpha },\widetilde{\omega }\right) $ we have the
orthogonal decomposition $f=\overline{\partial }\overline{\partial }_{\alpha
}^{\ast }\mathcal{N}_{\left( n,q\right) }^{\alpha }f+\overline{\partial }%
_{\alpha }^{\ast }\overline{\partial }\mathcal{N}_{\left( n,q\right)
}^{\alpha }f$ and $\overline{\partial }\mathcal{N}_{\left( n,q\right)
}^{\alpha }=\mathcal{N}_{\left( n,q+1\right) }^{\alpha }\overline{\partial }$%
, $\overline{\partial }_{\alpha }^{\ast }\mathcal{N}_{\left( n,q\right)
}^{\alpha }=\mathcal{N}_{\left( n,q-1\right) }^{\alpha }\overline{\partial }%
_{\alpha }^{\ast }$.

\item For every $\overline{\partial }$-closed form $f\in L_{\left(
n,q\right) }^{2}\left( X^{\pm },\widetilde{v}^{\alpha },\widetilde{\omega }%
\right) $, $\overline{\partial }\left( \overline{\partial }_{\alpha }^{\ast }%
\mathcal{N}_{\left( n,q\right) }^{\alpha }f\right) =f$.
\end{enumerate}
\end{corollary}

\begin{lemma}
\label{Cond moments H orthog}Let $f\in C_{\left( 0,q\right) }^{\infty
}\left( X\right) $, $1\leqslant q\leqslant n-1$, be a $\overline{\partial }$%
-closed form such that $f$ vanishes to infinite order on $L$. Let $\psi
_{1},\psi _{2}\in Dom\overline{\partial }\subset L_{\left( n,n-q\right)
}^{2}\left( X^{\pm },\widetilde{v}^{\alpha },\widetilde{\omega }\right) $
such that $\overline{\partial }\psi _{1}=\overline{\partial }\psi _{2}$. Then%
\begin{equation*}
\int_{X^{\pm }}f\wedge \left( \psi _{1}-\psi _{2}\right) =0.
\end{equation*}
\end{lemma}

\begin{proof}
By Corollary \ref{Neumann operator}, there exists $h\in L_{\left(
n,n-q-1\right) }^{2}\left( X^{\pm },\widetilde{v}^{\alpha },\widetilde{%
\omega }\right) $ such that $\psi _{1}-\psi _{2}=\overline{\partial }h$.
Since $f$ vanishes to infinite order on $L$ and $\widetilde{\omega }%
\leqslant \frac{C}{\widetilde{v}^{2}}\omega $, it follows that 
\begin{equation*}
\int_{X^{\pm }}f\wedge \left( \psi _{1}-\psi _{2}\right) =\underset{%
\varepsilon \rightarrow 0}{\lim }\int_{\left\{ v>\varepsilon \right\} \cap
X^{\pm }}f\wedge \overline{\partial }h=\underset{\varepsilon \rightarrow 0}{%
\lim }\left( \int_{\left\{ v>\varepsilon \right\} }\overline{\partial }%
f\wedge h+\int_{\left\{ v=\varepsilon \right\} \cap X^{\pm }}f\wedge
h\right) =0.
\end{equation*}
\end{proof}

\begin{proposition}
\label{sol dbar comp}Let $f\in C_{\left( 0,q\right) }^{\infty }\left(
X\right) $, $1\leqslant q\leqslant n-1$, be a $\overline{\partial }$-exact
form such that $f$ vanishes to infinite order on $L$. Then for every $\alpha
>0$, there exists $u\in L_{\left( 0,q-1\right) }^{2}\left( X^{\pm },%
\widetilde{v}^{-\alpha },\widetilde{\omega }\right) $ such that $\overline{%
\partial }u=f$ and $N_{-\alpha ,\widetilde{\omega },\widetilde{v}}\left(
u\right) \leqslant C_{\alpha }N_{-\alpha ,\widetilde{\omega },\widetilde{v}%
}\left( f\right) $, with $C_{\alpha }>0$ independent of $f$.
\end{proposition}

\begin{proof}
Step 1. Definition by duality of $u\in L_{\left( 0,q-1\right) }^{2}\left(
X^{\pm },\widetilde{v}^{-\alpha },\widetilde{\omega }\right) $, $1\leqslant
q\leqslant n-1$.

The proof of this point is inspired from \cite{Henkin00}, Proposition 5.3.
By Proposition \ref{Solution d bar k>>0}, $\mathcal{R}_{\left( n,q\right)
}^{\alpha }\left( X^{\pm }\right) $ is closed for every $\alpha >0$ and by
Corollary \ref{Neumann operator} we can find a bounded operator $T_{\left(
n,q\right) }^{\alpha }=\overline{\partial }_{\alpha }^{\ast }\mathcal{N}%
_{\left( n,q\right) }^{\alpha }:\mathcal{R}_{\left( n,q\right) }^{\alpha
}\left( X^{\pm }\right) \rightarrow L_{\left( n,q-1\right) }^{2}\left(
X^{\pm },\widetilde{v}^{\alpha },\widetilde{\omega }\right) $, such that $%
\overline{\partial }T_{\left( n,q\right) }^{\alpha }\varphi =\varphi $ for
every $\varphi \in \mathcal{R}_{\left( n,q\right) }^{\alpha }\left( X^{\pm
}\right) $, $1\leqslant q\leqslant n$.

Define now the continuous linear form $\Phi _{f}$ on $\mathcal{R}_{\left(
n,n-q+1\right) }^{\alpha }\left( X^{\pm }\right) $, $1\leqslant q\leqslant n$%
, by%
\begin{equation*}
\Phi _{f}\left( \varphi \right) =\int_{X^{\pm }}f\wedge T_{\left(
n,n-q+1\right) }^{\alpha }\varphi ,\ \varphi \in \mathcal{R}_{\left(
n,n-q+1\right) }^{\alpha }\left( X^{\pm }\right) .
\end{equation*}

By the Hahn-Banach theorem, we extend $\Phi _{f}$ as a linear form $%
\widetilde{\Phi _{f}}$ on $L_{\left( n,n-q+1\right) }^{2}\left( X^{\pm },%
\widetilde{v}^{\alpha },\widetilde{\omega }\right) $ such that $\left\Vert 
\widetilde{\Phi _{f}}\right\Vert =\left\Vert \Phi _{f}\right\Vert $. Since $%
\left( L_{\left( n,n-q+1\right) }^{2}\left( X^{\pm },\widetilde{v}^{\alpha },%
\widetilde{\omega }\right) \right) ^{\prime }=L_{\left( 0,q-1\right)
}^{2}\left( X^{\pm },\widetilde{v}^{-\alpha },\widetilde{\omega }\right) $
by the pairing 
\begin{equation*}
\left( \beta _{1},\beta _{2}\right) =\int_{X^{\pm }}\beta _{1}\wedge \beta
_{2},\ \beta _{1}\in L_{\left( 0,q-1\right) }^{2}\left( X^{\pm },\widetilde{v%
}^{-\alpha },\widetilde{\omega }\right) ,\ \beta _{2}\in L_{\left(
n,n-q+1\right) }^{2}\left( X^{\pm },\widetilde{v}^{\alpha },\widetilde{%
\omega }\right) ,
\end{equation*}%
there exists $u\in L_{\left( 0,q-1\right) }^{2}\left( X^{\pm },\widetilde{v}%
^{-\alpha },\widetilde{\omega }\right) $ such that 
\begin{equation*}
\widetilde{\Phi _{f}}\left( \varphi \right) =\int_{X^{\pm }}u\wedge \varphi
\end{equation*}%
for every $\varphi \in \mathcal{R}_{\left( n,n-q+1\right) }^{\alpha }\left(
X^{\pm }\right) $.

Step 2. We prove that $\overline{\partial }\left( -1\right) ^{q}u=f$, $%
1\leqslant q\leqslant n-1$.

Let $\varphi =\overline{\partial }\psi \in C_{\left( n,n-q+1\right)
}^{\infty }\left( X^{\pm }\right) $ with $\psi \in \mathcal{D}_{\left(
n,n-q\right) }\left( X^{\pm }\right) $. Set $g_{\alpha }=\overline{\partial }%
_{\alpha }^{\ast }\mathcal{N}_{\left( n,n-q+1\right) }^{\alpha }\overline{%
\partial }\psi \in L_{\left( n,n-q\right) }^{2}\left( X^{\pm },\widetilde{v}%
^{\alpha },\widetilde{\omega }\right) $. By Corollary \ref{Neumann operator}%
, $\overline{\partial }g_{\alpha }=\varphi $ and by Lemma \ref{Cond moments
H orthog} 
\begin{equation}
\int_{X^{\pm }}f\wedge g_{\alpha }=\int_{X^{\pm }}f\wedge \psi .
\label{alfa}
\end{equation}%
But by step 1 we have%
\begin{equation}
\widetilde{\Phi _{f}}\left( \varphi \right) =\int_{X^{\pm }}u\wedge 
\overline{\partial }\psi =\Phi _{f}\left( \varphi \right) =\int_{X^{\pm
}}f\wedge g_{\alpha }  \label{beta}
\end{equation}

and by (\ref{alfa}) and (\ref{beta}) it follows that%
\begin{equation*}
\int_{X^{\pm }}f\wedge \psi =\int_{X^{\pm }}u\wedge \overline{\partial }\psi
\end{equation*}%
for every $\psi \in \mathcal{D}_{\left( n,n-q\right) }\left( X^{\pm }\right) 
$. Therefore $\overline{\partial }\left( -1\right) ^{q}u=f$ and the
Proposition is proved.
\end{proof}

\begin{remark}
\label{Tilde} Since $\omega \leqslant \widetilde{\omega }\leqslant \frac{C}{%
\widetilde{v}^{2}}\omega $, by Lemma VIII.6.3 of \cite{Demaillylivre} it
follows that:

a) Let $f$ be a smooth $\left( n,q\right) $-form on $X$ such that $f$
vanishes to order $k$ on $L$. Then $f\in L_{\left( n,q\right) }^{2}\left(
X^{\pm },\widetilde{v}^{-k},\widetilde{\omega }\right) $

Indeed%
\begin{equation*}
\int_{X^{\pm }}\left\vert f\right\vert _{\widetilde{\omega }}^{2}\widetilde{v%
}^{-k}dV_{\widetilde{\omega }}\leqslant \int_{X^{\pm }}\left\vert
f\right\vert _{\omega }^{2}\widetilde{v}^{-k}dV_{\omega }<\infty
\end{equation*}

b) Let $f\in L_{\left( n,q\right) }^{2}\left( X^{\pm },\widetilde{v}^{-k},%
\widetilde{\omega }\right) $, $k>2$. Then $f\in L_{\left( n,q\right)
}^{2}\left( X^{\pm },\widetilde{v}^{-k+2},\omega \right) $.

Indeed%
\begin{equation*}
\int_{X^{\pm }}\left\vert f\right\vert _{\omega }^{2}\widetilde{v}%
^{-k+2}dV_{\omega }\leqslant C\int_{X^{\pm }}\left\vert f\right\vert _{%
\widetilde{\omega }}^{2}\widetilde{v}^{-k}dV_{\widetilde{\omega }}<\infty ,\
C>0.
\end{equation*}
\end{remark}

\section{Nonexistence of Levi flat hypersurfaces}

\begin{proposition}
\label{Extension}Let $L$ be a compact $C^{\infty }$ Levi flat hypersurface
in a K\"{a}hler manifold $X$ of dimension $n\geqslant 3$ such that the
normal bundle $\mathcal{N}_{L}^{1,0}$ to the Levi foliation admits a $C^{2}$
Hermitian metric with leafwise positive curvature. Let $u\in C_{\left(
0,q\right) }^{\infty }\left( L\right) $, $1\leqslant q\leqslant n-2$, such
that $\overline{\partial }_{b}u=0$. Then for every $k\in \mathbb{N}^{\ast }$
there exist a $\overline{\partial }$-closed extension $U_{k}\in C_{\left(
0,q\right) }^{k}\left( X\right) $ of $u$.
\end{proposition}

\begin{proof}
By Proposition \ref{Strongly pseudoconvex} there exist a neighborhood $U$ of 
$L$, $c>0$ and a non-negative function $v\in C^{2}\left( \overline{U}\right) 
$ vanishing on $L$ such that $v=g\delta _{L}^{2}$ and $-i\partial \overline{%
\partial }\ln v\geqslant c\omega $ on $U\backslash L$. Let $\widetilde{u}\in
C_{\left( 0,q\right) }^{\infty }\left( X\right) $ be an extension of $u$
such that $\overline{\partial }\widetilde{u}$ vanishes to infinite order on $%
L$. Since $\overline{\partial }\widetilde{u}\in L_{\left( 0,q+1\right)
}^{2}\left( X^{\pm },\delta _{L}^{-2k},\omega \right) $, $q+1\leqslant n-1$
and $L_{\left( 0,q\right) }^{2}\left( X^{\pm },\delta _{L}^{-2k},\omega
\right) =L_{\left( 0,q\right) }^{2}\left( X^{\pm },\widetilde{v}^{-k},\omega
\right) $ for every $k\in \mathbb{N}$, by Remark \ref{Tilde} a) and
Proposition \ref{sol dbar comp} it follows that for every $k\in \mathbb{N}%
^{\ast }$ there exist a Hermitian complete metric $\widetilde{\omega }$ on $%
X\backslash L$, $\omega \leqslant \widetilde{\omega }\leqslant \frac{C}{v^{2}%
}\omega $ and $h^{\pm }\in L_{\left( 0,q\right) }^{2}\left( X^{\pm },\delta
_{L}^{-2k},\widetilde{\omega }\right) $ such that $\overline{\partial }%
h^{\pm }=\overline{\partial }\widetilde{u}$ on $X^{\pm }$. By Remark \ref%
{Tilde} b) we have $h^{\pm }\in L_{\left( 0,q\right) }^{2}\left( X^{\pm
},\delta _{L}^{-2k+4},\omega \right) $. So by using Theorem \ref%
{Brinkschulte}, for $k$ big enough we can choose $h^{\pm }\in C_{\left(
0,q\right) }^{s\left( k\right) }\left( \overline{X^{\pm }}\right) $, $%
s\left( k\right) \underset{k\rightarrow \infty }{\thicksim }\sqrt{k}$. This
means that for $k$ big enough, the form $h$ defined as $h^{\pm }$ on $%
\overline{X^{\pm }}$ is of class $C^{k}$ on $X$ and vanishes on $L$. So $%
U_{k}=\widetilde{u}-h^{\pm }$ is a $C^{k}$-smooth $\overline{\partial }$%
-closed form on $X$ which is an extension of $u$.
\end{proof}

\begin{theorem}
\label{Nonexist}Let $X$ be a compact connected K\"{a}hler manifold of
dimension $n\geqslant 3$ and $L$ a $C^{\infty }$ compact Levi flat
hypersurface. Then the normal bundle to the Levi foliation does not admit
any Hermitian metric of class $C^{2}$ with leafwise positive curvature.
\end{theorem}

\begin{proof}
Suppose that the normal bundle $\mathcal{N}$ to the Levi foliation admits a
Hermitian metric of class $C^{2}$ with leafwise positive curvature. Since $%
\mathcal{N}$ is topologically trivial, its curvature form $\Theta ^{\mathcal{%
N}}$ for the K\"{a}hler metric of $X$ is $d$-exact. So there exists a $1$%
-form $u$ of class $C^{\infty }$ on $L$ such that $du=\Theta ^{\mathcal{N}}$%
; we may suppose that $u$ is real and $u=u^{0,1}+\overline{u^{0,1}}$, where $%
u^{0,1}$ is the $\left( 0,1\right) $ component of $u$. Since $\Theta ^{%
\mathcal{N}}$ is a $\left( 1,1\right) $-form, it follows that $\overline{%
\partial }_{b}u^{0,1}=0$, where $\overline{\partial }_{b}$ is the tangential
Cauchy-Riemann operator. By Proposition \ref{Extension} there exists a $%
C^{k} $-extension $U^{0,1}$ of $u^{0,1}$ to $X$, $k\geqslant 2$, such that $%
\overline{\partial }U^{0,1}=0$.

By Hodge symmetry and Dolbeault isomorphism $H^{0,1}\left( X,\mathbb{C}%
\right) \thickapprox \overline{H^{1,0}\left( X,\mathbb{C}\right) }%
\thickapprox \overline{H^{0}\left( X,\Omega _{X}^{1}\right) }$, where $%
\Omega _{X}^{1}$ is the sheaf of holomorphic $1$-forms on $X$. So there
exists $\eta \in H^{0}\left( X,\Omega _{X}^{1}\right) $ and $\Phi \in
C^{k}\left( X\right) $ such that $\widetilde{U^{0,1}}=\overline{\eta }+%
\overline{\partial }\Phi $. It follows that $\Theta ^{\mathcal{N}}=i\partial
_{b}\overline{\partial _{b}}\func{Im}\Phi $ on $L$ and this gives a
contradiction at the point of $L$ where $\func{Im}\Phi $ reaches its maximum.
\end{proof}

\begin{remark}
A first version of this paper was announced on arXiv in 2014, but there was
a gap in the proofs of \S 4, which is now corrected. Recently, Brinkschulte
proved a generalization of Theorem \ref{Nonexist} for compact Levi flat
hypersurfaces in complex manifolds (see Theorem 1.1. of \cite%
{Brinkschulte2018}). She uses crucially the Proposition 4.1\ of \cite%
{Brinkschulte2018}, whose statement and proof are the same as Proposition %
\ref{Strongly pseudoconvex} of this paper and which are unchanged from 2014
in our preprint arXiv:1406.5712. However she refers only to Proposition 1.1
of \cite{Ohsawa2013}, where the lower positive bound for the eigenvalues of
the strongly plurisubharmonic function is not mentioned.
\end{remark}

\begin{acknowledgement}
We would like to thank M. Adachi and T.-C. Dinh for very useful discussions.
We would also thank the referees for their remarks.
\end{acknowledgement}

\renewcommand\baselinestretch{1} 

\end{document}